\documentclass[12pt]{amsart}

\usepackage{amsmath, amsfonts, amsthm, amssymb}
\usepackage{url}
\usepackage{graphicx}

\usepackage{caption}	
\usepackage{amscd}	
\usepackage{subcaption} 
\usepackage{tikz}
\usetikzlibrary{plotmarks}	% Lets you have different shaped nodes as the trees
\usepackage{stix}  %Let's you use \circlerighthalfblack for the node that is half-white & half-black
\usepackage{xfrac} % Allows you to make slanty line fractions via the command \sfrac{}{}  WARNING: sometimes using this brings up Warning errors (eg. fontsize conflict).  I am not sure why.
\usepackage{mathdots}

\usepackage[colorlinks=true, citecolor=blue, linktocpage]{hyperref}

\usepackage{enumerate}	%Lets you enumerate lists with things other than numbers eg. (a) or (i)'s

\DeclareMathOperator{\ggcd}{gcd}

\newcommand{\tikzcircle}[2][red,fill=red]{\tikz[baseline=-0.5ex]\draw[#1,radius=#2] (0,0) circle ;}%

\theoremstyle{theorem}
\newtheorem{theorem}{Theorem}
\newtheorem{proposition}[theorem]{Proposition}
\theoremstyle{definition}
\newtheorem{definition}{Definition}

\newtheorem{example}[theorem]{Example}

\begin{document}

\title{Lattice Point Visibility on Generalized Lines of Sight}
%\markright{Generalized lattice point visibility}
\author{Edray H. Goins, Pamela E. Harris, Bethany Kubik, and Aba Mbirika}

\date{\today}


\begin{abstract}
For a fixed $b\in\mathbb{N}=\{1,2,3,\ldots\}$ we say that a point $(r,s)$ in the integer lattice
$\mathbb{Z} \times \mathbb{Z}$ is $b$-visible from the origin if it lies on the graph of a power function $f(x)=ax^b$ with $a\in\mathbb{Q}$ and no other integer lattice point lies on this curve (i.e., line of sight) between $(0,0)$ and $(r,s)$.
We prove that the proportion of $b$-visible integer lattice points is given by $1/\zeta(b+1)$, where $\zeta(s)$ denotes the Riemann zeta function. We also show that even though the proportion of $b$-visible lattice points approaches $1$ as $b$ approaches infinity,  there exist arbitrarily large rectangular arrays of $b$-invisible lattice points for any fixed $b$.  This work specialized to $b=1$ recovers original results from the classical lattice point visibility setting where the lines of sight are given by linear functions with rational slope through the origin. 
\end{abstract}

\maketitle

%\noindent
%The \textit{American Mathematical Monthly} style incorporates the following \LaTeX\ packages.  These styles should \textit{not} be included in the document header.
%\begin{itemize}
%\item times
%\item pifont
%\item graphicx
%\item color
%\item AMS styles: amsmath, amsthm, amsfonts, amssymb
%\item url
%\end{itemize}
%Use of other \LaTeX\ packages should be minimized as much as possible. Math notation, like $c = \sqrt{a^2 +b^2}$, can be left in \TeX's default Computer Modern typefaces for manuscript preparation; or, if you have the appropriate fonts installed, the \texttt{mathtime} or \texttt{mtpro} packages may be used, which will better approximate the finished article.
%
%Web links can be embedded using the \verb~\url{...}~ command, which will result in something like \url{http://www.maa.org}.  These links will be active and stylized in the online publication.

\section{Introduction}
A point $(r,s)$ in the integer lattice $\mathbb{Z} \times \mathbb{Z}$ is said to be visible from the origin if it lies on a straight line through the origin $(0,0)$ and no other lattice point lies on this line of sight between $(0,0)$ and $(r,s)$. Given this definition, it is natural to ask what proportion of lattice points are visible from the origin, which is equivalent to computing the probability that two integers are relatively prime. This problem was first addressed in the 1800s by numerous people including: Dirichlet, who proved a weaker form of the problem in 1849~\cite{Dir1849}; Ces\`aro, who is often attributed as having posed this problem in 1881~\cite{Ces1881}; and Sylvester, who along with Ces\`aro gave independent proofs of this result in 1883~\cite{Ces1883, Syl1883}.
Ces\`{a}ro proved that the probability that two randomly chosen integers in $\{1,2,\ldots, n\}$ are coprime is given by $1/\zeta(2)$ as $n$ approaches infinity, where $\zeta(s) = \sum_{n=1}^\infty 1/n^s$ denotes the Riemann zeta function~\cite{Ces1881}. Thus, the proportion of visible integer lattice points is given by $1/\zeta(2) = 6/\pi^2 \approx .608$.

In 1971, Herzog and Stewart characterized patterns of visible (respectively, invisible) points within the approximately 60\%  (respectively, 40\%) of the lattice containing visible (respectively, invisible) points~\cite{HerzogStewart} and their seminal work continues to motivate research in this area~\cite{Adhikari, Adhikari.vip, Chen, Mbirika2015, Laishram, Laison, Nicholson}. 
Additionally, it has been shown that the set of lattice points in the plane visible from the origin contains arbitrarily large square arrays of adjacent invisible lattice points \cite[Theorem 5.29, p. 119]{Apostol}. This is connected to a celebrated result in number theory regarding the existence of two mutually pairwise coprime sets of consecutive integers. Since then, others have further studied properties of strings of consecutive composite numbers and their connection to integer lattice point visibility~\cite{De, Eggleton, Schumer}.

\begin{figure}[h!]
\begin{center}
\resizebox {.6\textwidth} {!} {
\begin{tikzpicture}[yscale=.1,xscale=2]
		
\draw [<->, ultra thick, blue] (0,85) -- (0,0) -- (6.25,0);

\foreach \x in {1,...,6}
		\draw [very thin] (\x,0) -- (\x,80);

\foreach \x in {1,...,6}
	\node [above] at (\x,\x*10) {};%{\includegraphics[width=3.5ex]{tree}};

\foreach \x in {1,...,6}
	\node [above] at (\x,2*\x*\x) {};%{\includegraphics[width=3.5ex]{tree}};

\foreach \x in {10,20,...,70,80}
		\draw [very thin] (0,\x) -- (6,\x);

\draw[very thick,red] (0,0) -- (6,60);
\draw[very thick,red] (0,0) parabola (6,72);

\node at (1,10) {\tikzcircle[fill=white]{2.5pt}};
\foreach \x in {2,...,6}
	\node at (\x,\x*10) {\tikzcircle[fill=black]{2.5pt}};

\node at (1,2) {\tikzcircle[fill=white]{2.5pt}};
\foreach \x in {2,...,6}
	\node at (\x,2*\x*\x) {\tikzcircle[fill=black]{2.5pt}};

\node [above left] at (1,10) {\small(1,10)};
\node [above left] at (2,20) {\small(2,20)};
\node [above left] at (3,30) {\small(3,30)};
\node [above left] at (4,40) {\small(4,40)};
\node [above left] at (5,50) {\small(5,50)};
\node [above left] at (6,60) {\small(6,60)};

\node [above left] at (1,2) {\small(1,2)};
\node [above left] at (2,8) {\small(2,8)};
\node [above left] at (3,18) {\small(3,18)};
\node [above left] at (4,32) {\small(4,32)};
\node [above left] at (6,72) {\small(6,72)};

\foreach \x in {1,...,6}
		\node [below] at (\x,0) {\x};
\foreach \x in {10,20,...,80}
		\node [left] at (0,\x) {\x};
\node at (0,0) {\includegraphics[width=8ex]{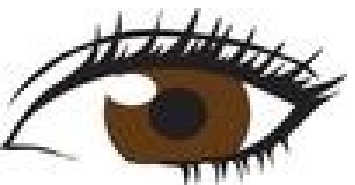}};
\end{tikzpicture}}
\end{center}
\caption{Lines of sight $f(x) = 10x$ and $g(x) = 2 x^2$ with visible and invisible points.}
\label{fig:two_lines_of_sights}
\end{figure}

In this work, we fix $b\in\mathbb{N}$ and say that a point $(r,s)$ in the integer lattice $\mathbb{Z} \times \mathbb{Z}$ is $b$-visible from the origin if it lies on the graph of a power function $f(x)=ax^b$ with $a\in\mathbb{Q}$ and no other integer lattice point lies on this curve (i.e., line of sight) between $(0,0)$ and $(r,s)$. Hence, 
our work specialized to $b=1$ recovers the classical setting of lattice point visibility whose lines of sight are given by linear functions $f(x) = ax$ with $a \in \mathbb{Q}$. We remark that throughout this work, following the wording introduced by P\'olya, we often refer to lattice points as trees and collections of adjacent trees as forests~\cite{Allen, Polya}.

Figure~\ref{fig:two_lines_of_sights} contains two examples of lines of sight 
on which we mark the lattice points that are visible with white nodes and those that are invisible with black nodes. Figure \ref{fig:InvisibleGrid} marks %, with a black circle, 
the $b$-invisible lattice points in the square $[0,50]\times[0,50]$ for $b=1,2,3,4$.  Note that the number of $b$-visible points increases substantially relative to a small growth in $b$
even in this small portion of the integer lattice. This observation, presented in Table~\ref{table:observed_and_expected_data}, leads us naturally to our first result. 
\begin{figure}[h!]
    \centering
    \begin{subfigure}[t]{0.4\textwidth}
        \centering
        \includegraphics[width=\textwidth]{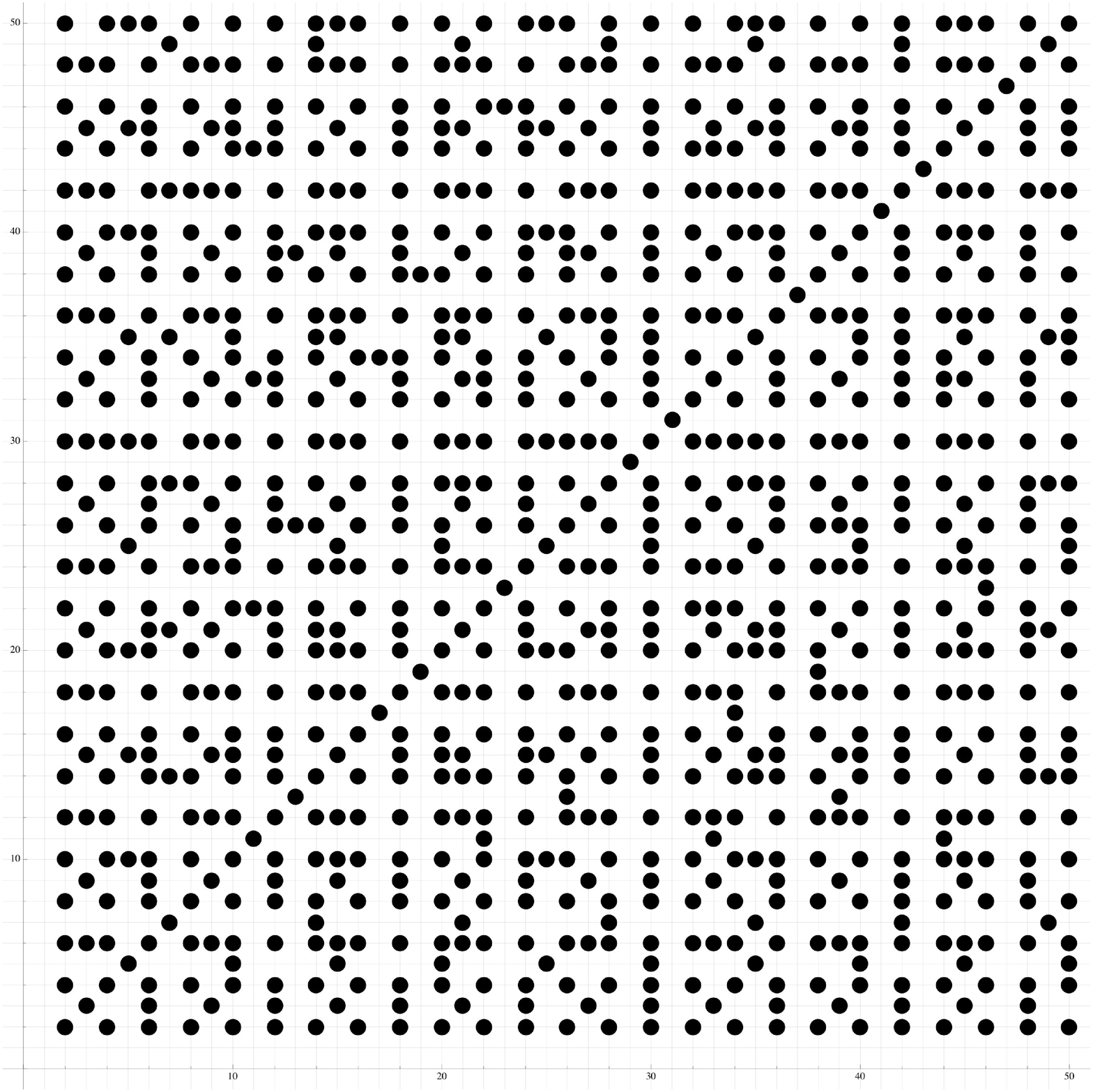}
        \caption{$b=1$}
    \end{subfigure}%
    \hspace{.5in} 
    \begin{subfigure}[t]{0.4\textwidth}
        \centering
        \includegraphics[width=\textwidth]{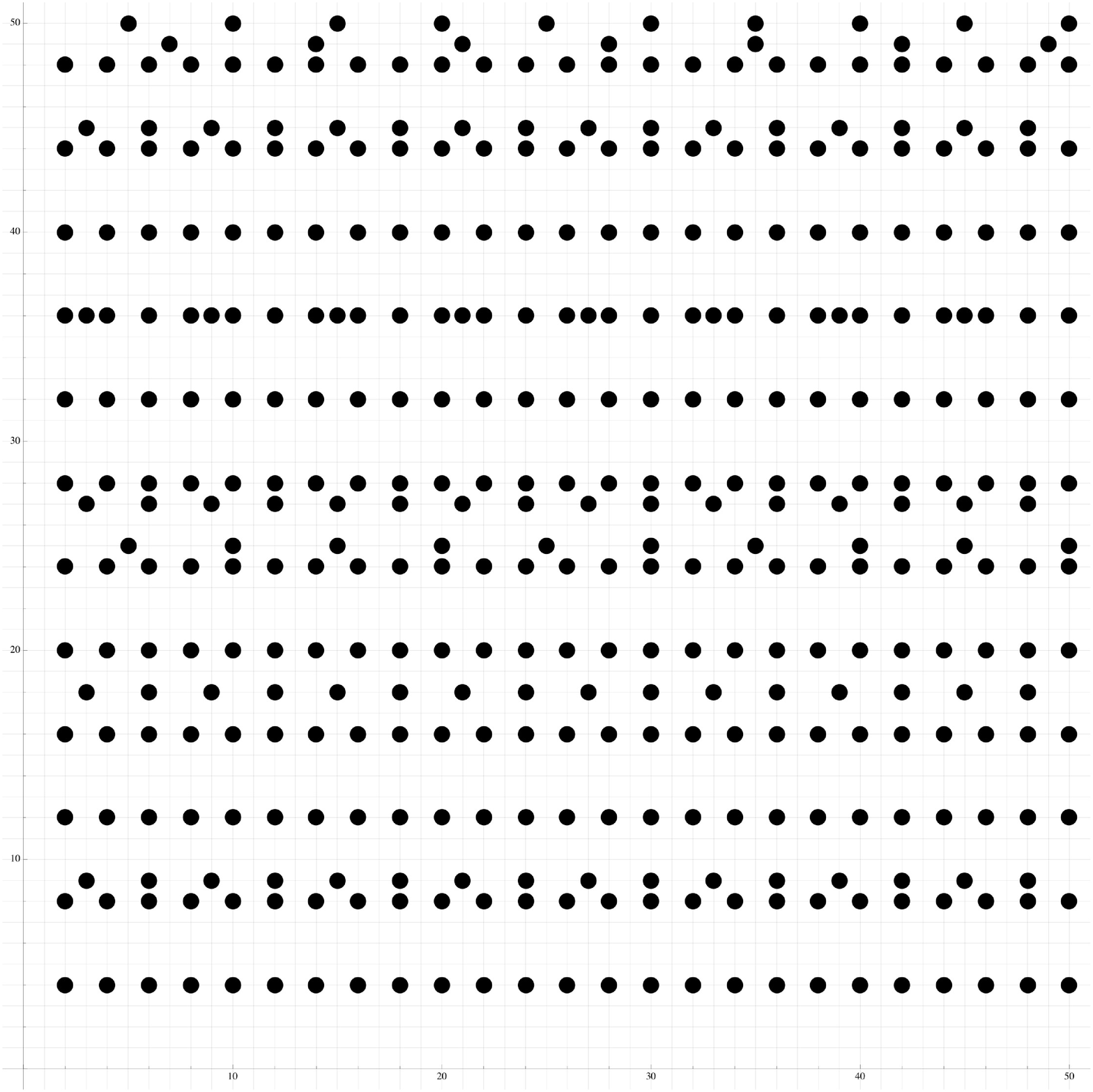}
        \caption{$b=2$}
    \end{subfigure}\\
    \begin{subfigure}[t]{0.4\textwidth}
        \centering
        \includegraphics[width=\textwidth]{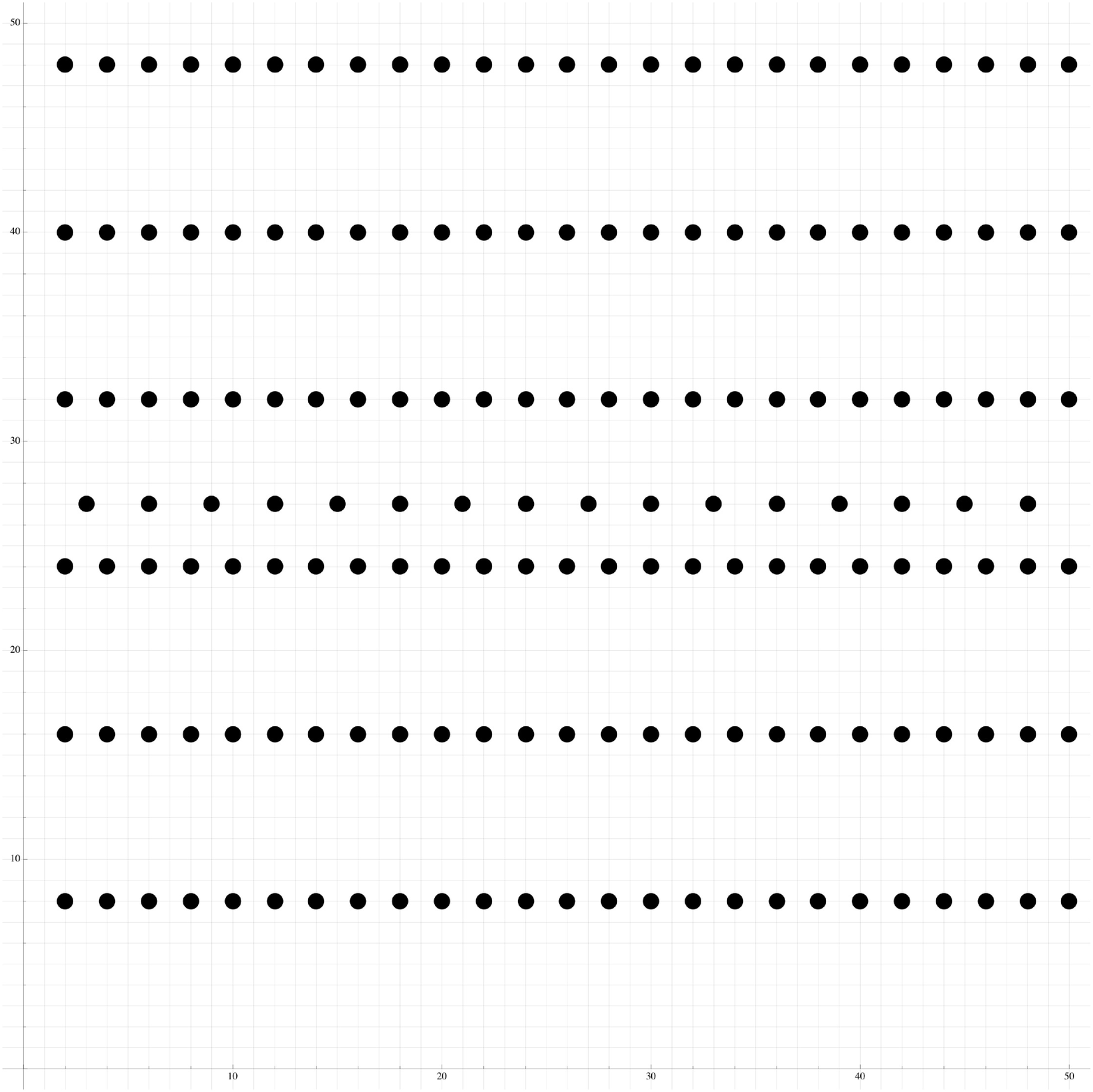}
        \caption{$b=3$}
    \end{subfigure}
    \hspace{.5in}
        \begin{subfigure}[t]{0.4\textwidth}
        \centering
        \includegraphics[width=\textwidth]{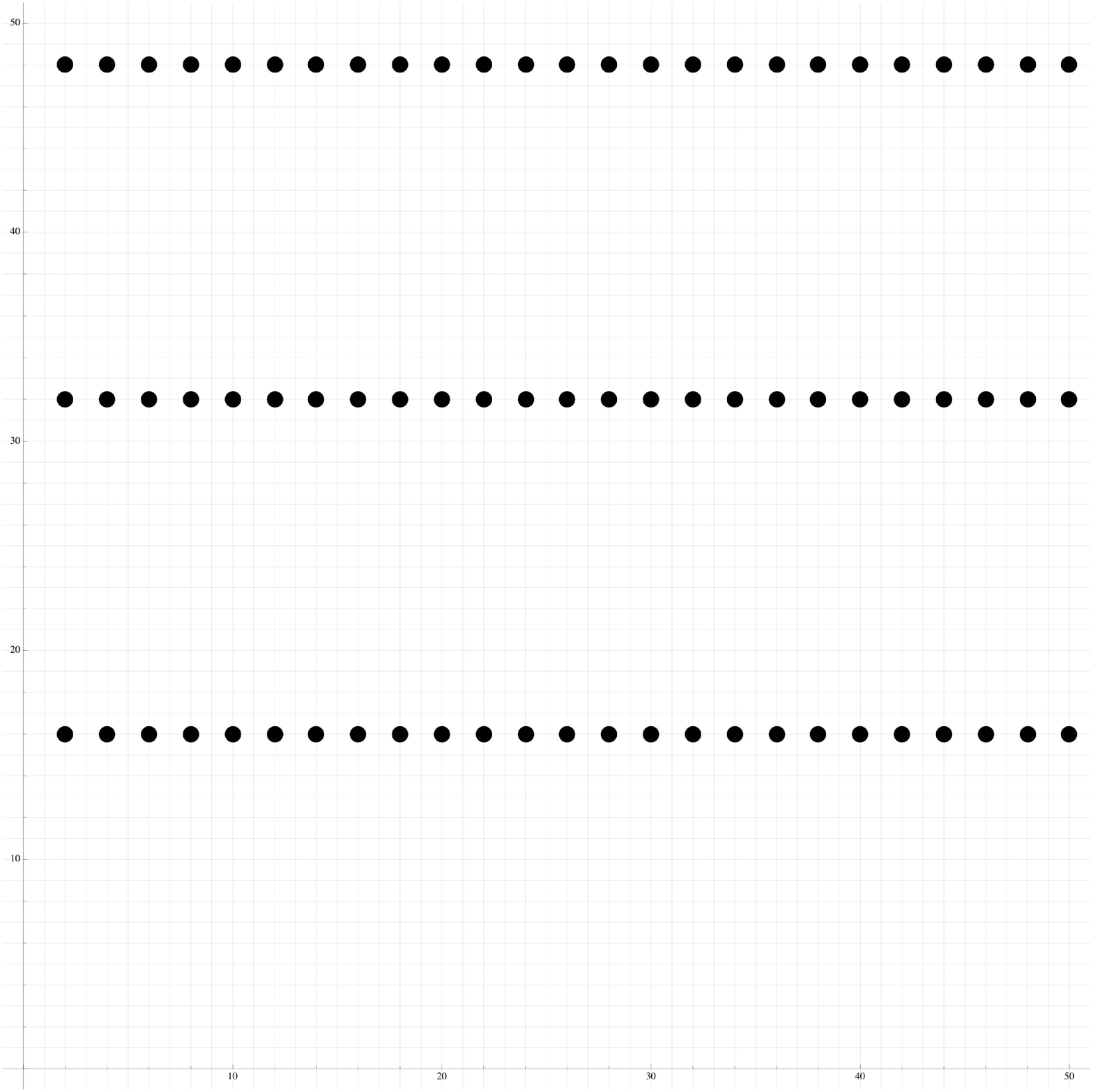}
        \caption{$b=4$}
    \end{subfigure}
    \caption{The $b$-invisible lattice points in $[0,50]\times[0,50]$ when $b=1,2,3,4$.}\label{fig:InvisibleGrid}
\end{figure}

\begin{table}[h!]
\begin{center}
\caption{Proportion of $b$-visible and $b$-invisible points for $b=1,2,3,4$ with all values approximated to 3 decimal places.}\label{table:observed_and_expected_data}

\begin{tabular}{|c||c|c|c|c|}
\hline
$b$ & $\zeta(b+1)$ & $\frac{1}{\zeta(b+1)}$ & $1 - \frac{1}{\zeta(b+1)}$ & Proportion of $b$-invisible points in $50 \times 50$ grid \\ \hline \hline
1 & 1.644 & .608 & \textbf{.392} & $953/2500\approx\mathbf{.381}$ \\
\hline
2 & 1.202 & .832 & 
\textbf{.168}& ${399}/{2500}\approx\mathbf{.160}$ \\
\hline
3 & 1.082 & .924 & 
\textbf{.076}& ${166}/{2500}\approx\mathbf{.066}$ \\
\hline
4 & 1.036 & .964 &  
\textbf{.035}& ${\phantom{0}75}/{2500}\approx\mathbf{.030}$ \\ 
\hline
\end{tabular}
\end{center}
\end{table}

\begin{theorem}\label{thm:proportion_of_b_visible_points}
Fix an integer $b\in\mathbb{N}$.  Then the proportion of points $(r,s)\in\mathbb{N}\times\mathbb{N}$ that are $b$-visible is $\displaystyle\frac{1}{\zeta(b+1)}$.
\end{theorem}

Theorem~\ref{thm:proportion_of_b_visible_points} implies that the proportion of $b$-visible lattice points approaches $1$ as $b$ approaches infinity. However, as our next result shows, for any fixed $b\in\mathbb{N}$ there exist arbitrarily large $b$-invisible rectangular forests, that is, rectangular arrays of adjacent $b$-invisible  integer lattice points.

\begin{theorem}\label{thm:invisible_forests}
Let $b\in\mathbb{N}$. For any integers $n,m > 0$, there exists a lattice point $(r, s)$ such that every point
$(r+i,s+j)$, where $0\leq i< n $ and $0\leq j< m$, is $b$-invisible from the origin.\end{theorem}

Although we present a proof that arbitrarily large $b$-invisible rectangular forests exist for all values $b\in\mathbb{N}$, our work does not construct forests close to the origin. In the classical $b=1$ case, the work of Herzog and Stewart used prime matrices and the Chinese remainder theorem to compute invisible square forests and they presented $2\times 2$ and $3\times 3$ invisible forests shown in Figure \ref{fig:closetrees}~\cite{HerzogStewart}.

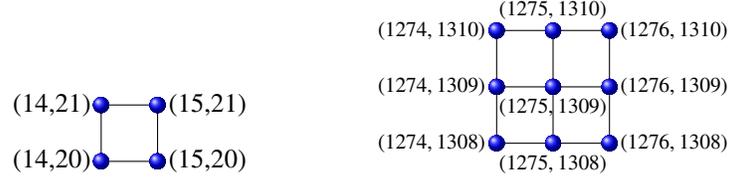
\begin{figure}[h!]
    \centering
\begin{tikzpicture}[scale=.75]
\draw (0,0)--(0,1)--(1,1)--(1,0)--(0,0);
\foreach \x in {0,1}
		\shade[ball color=blue] (\x,0) circle (.75ex);
\foreach \x in {0,1}
		\shade[ball color=blue] (\x,1) circle (.75ex);
\node [left] at (0,0) {\footnotesize (14,20)};
\node [left] at (0,1) {\footnotesize (14,21)};
\node [right] at (1,0) {\footnotesize (15,20)};
\node [right] at (1,1) {\footnotesize (15,21)};
\end{tikzpicture}
    \hspace{.5in} 
\begin{tikzpicture}[scale=.75]
\draw (0,0)--(0,2)--(2,2)--(2,0)--(0,0);
\draw (0,1)--(2,1);
\draw (1,0)--(1,2);
\foreach \x in {0,...,2}
		\shade[ball color=blue] (\x,0) circle (.75ex);
\foreach \x in {0,...,2}
		\shade[ball color=blue] (\x,1) circle (.75ex);
\foreach \x in {0,...,2}
		\shade[ball color=blue] (\x,2) circle (.75ex);
\node [left] at (0,0) {\tiny $(1274,1308)$};
\node [left] at (0,1) {\tiny $(1274,1309)$};
\node [left] at (0,2) {\tiny $(1274,1310)$};
\node [below] at (1,0) {\tiny $(1275,1308)$};
\node [below] at (1,1) {\tiny $(1275,1309)$};
\node [above] at (1,2) {\tiny $(1275,1310)$};
\node [right] at (2,0) {\tiny $(1276,1308)$};
\node [right] at (2,1) {\tiny $(1276,1309)$};
\node [right] at (2,2) {\tiny $(1276,1310)$};
\end{tikzpicture}
\caption{The $2\times2$ and $3\times 3$ invisible forests lying closest to the origin.}\label{fig:closetrees}
\end{figure}
It is easily verified that every point $(r,s)$ in the forests of Figure~\ref{fig:closetrees} satisfies the condition $\gcd(r,s) > 1$. It turns out that, up to symmetry, these are the closest invisible square forests of size $n=2$ and $n=3$. 
In a brief remark, Wolfram claims to have found the closest $4 \times 4$ invisible forest, being located approximately 12 million units from the origin~\cite[p.\ 1093]{Wolfram}. However, this has yet to be confirmed in the literature. Although to date, no one knows the closest $n \times n$ invisible square forests for $n \geq 5$, recently bounds have been given on where invisible square forests might exist in the integer lattice~\cite{Laishram,Pigh02}. Finding such bounds in our generalized setting remains an open problem.

Our paper is organized as follows. Section~\ref{Section2} contains the necessary definitions to make our approach precise. Section~\ref{Section3} provides a proof of Theorem~\ref{thm:proportion_of_b_visible_points}.  Section~\ref{Section4} gives a construction of arbitrarily large rectangular  $b$-invisible forests, thereby proving Theorem~\ref{thm:invisible_forests}.

%%%%%%%%%%%%%%%%%%%%%%%%%%%%%%%%%%%%%%%%%%%%%%%%%%%%%%%%%%%%%%%%%%
%%%%%%%%%%%%%%%%%%%%%%%%%%%%%%%%%%%%%%%%%%%%%%%%%%%%%%%%%%%%%%%%%%
%%%%%%%%%%%%%%%%%%         SECTION 2       %%%%%%%%%%%%%%%%%%%%%%%
%%%%%%%%%%%%%%%%%%%%%%%%%%%%%%%%%%%%%%%%%%%%%%%%%%%%%%%%%%%%%%%%%%
%%%%%%%%%%%%%%%%%%%%%%%%%%%%%%%%%%%%%%%%%%%%%%%%%%%%%%%%%%%%%%%%%%

\section{Background}\label{Section2}
 
The results presented in this paper are limited to the first quadrant of the plane, and, due to the symmetry of the plane, our results can be easily extended to apply to all of $\mathbb{Z}\times\mathbb{Z}$.

\begin{definition}\label{def:visible} Fix $b\in \mathbb{N}$. A point $(r,s)\in \mathbb{N}\times\mathbb{N}$ is said to be \textit{$b$-invisible} if the following two conditions hold:
\begin{enumerate}[(1)]
\item The point $(r,s)$ lies on the graph of $f(x)=ax^b$ for some $a\in \mathbb{Q}$.  That is, $s=ar^b$.
\item There exists an integer $k>1$ such that $k$ divides $r$ and $k^b$ divides $s$.
\end{enumerate}
The point is said to be \textit{$b$-visible} if it satisfies Condition~1, but fails to satisfy Condition~2.
\end{definition} 

When we say that a point is $b$-invisible or $b$-visible, it is always with respect to the origin. If $(r,s)\in \mathbb{N}\times\mathbb{N}$ is $b$-invisible  and Condition 1 is satisfied by the function $f(x)=ax^b$, then $(-r,s)$, $(-r,-s)$, and $(r,-s)$ are $b$-invisible  under the functions $a(-x)^b$, $-a(-x)^b$, and $-ax^b$, respectively, and likewise for $b$-visible points.
Thus in our study it suffices to determine the $b$-visibility (meaning, whether the point is $b$-visible or $b$-invisible) of the lattice points in $\mathbb{N}\times\mathbb{N}$. 

To speak about the $b$-visibility of a lattice point in this new setting, we develop a generalization of the greatest common divisor.% defined below.

\begin{definition}\label{def:ggcd}
Fix $b \in \mathbb{N}$.   
The \textit{generalized greatest common divisor} of $r$ and $s$ with respect to $b$ is denoted $\ggcd_b$ and is defined as
$$\ggcd_b(r,s) := \max\{k \in \mathbb{N} \mid k \textrm{ divides } r \textrm{ and } k^b \textrm{ divides } s\}.$$
\end{definition}

Observe that $\ggcd_b$ coincides with the classical greatest common divisor when $b$ equals $1$. Moreover, from the lattice point visibility language, the new generalized greatest common divisor implies that for a fixed $b\in\mathbb{N}$ the point $(r,s)$ is \textit{$b$-visible} if there exists a function $f(x)=ax^b$ with $a\in\mathbb{Q}$ such that $(r,s)$ is on  the graph of $f$ and is the first integral point on the graph of $f$ from the origin. The following result gives a necessary and sufficient condition to determine $b$-visibility.

\begin{proposition}\label{prop:visibility_criterion}
A point $(r,s)\in \mathbb{N}\times\mathbb{N}$ is $b$-visible if and only if $\ggcd_b(r,s)=1$.
\end{proposition}

\begin{proof}
By Definition~\ref{def:visible}, a point $(r,s)\in \mathbb{N}\times\mathbb{N}$ is $b$-visible if $s=ar^b$ for some $a\in \mathbb{Q}$ and there does not exist an integer $k>1$ such that $k$ divides $r$ and $k^b$ divides $s$. Hence the largest positive integer that satisfies the visibility criterion is $1$. Thus $\ggcd_b(r,s)=1$.

For the other direction, suppose that $\ggcd_b(r,s)=1$.  Then $k=1$ is the largest integer such that $k$ divides $r$ and $k^b$ divides $s$ and the point $(r,s)$ does not satisfy Condition 2 of  Definition~\ref{def:visible}. Also, note that for every pair $(r,s)$, there exists a unique $a=s/r^b\in \mathbb{Q}$ such that $s=ar^b$.  Hence $(r,s)$ is $b$-visible.
\end{proof}

Note that in the classical $b=1$ setting of lattice point visibility, a point $(r,s)$ is visible if and only if $\gcd(r,s)=1$. Hence, 
Proposition~\ref{prop:visibility_criterion} generalizes the condition for a lattice point to be $b$-visible via the generalized greatest common divisor $\ggcd_b$ as stated in Definition~\ref{def:ggcd}. We also remark that the same integer lattice point can be $b$-visible and $b'$-invisible for distinct $b$ and $b'$. We illustrate this in the following example.

\begin{example} In Figure \ref{fig:exampledifferentcurves} the dotted curve is $f(x)=7x$, the dashed curve is $g(x)=x^2$, and the solid curve is $h(x)=\frac{1}{7}x^3$. A white node denotes a visible point, while a black node denotes an invisible point. In particular, the white-black point at $(7,49)$ is not 1-visible since $\gcd(7,49)=7$ and is not 2-visible since $\ggcd_2(7,49)=7$. However it is 3-visible since $\ggcd_3(7,49)=1$.

\begin{figure}
\centering
\begin{tikzpicture}[yscale=.07,xscale=1.1, domain=0:7.883735]  %These domain values are here just for the graph of the cubic polynomial (1/7)*x^3

\foreach \x in {1,...,8}
		\draw [very thin] (\x,0) -- (\x,70);

\foreach \x in {7,14,...,63,70}
		\draw [very thin] (0,\x) -- (8,\x);
		
\draw [<->, ultra thick, blue] (0,75) -- (0,0) -- (8.25,0);

\draw[dotted,very thick,red] (0,0) -- (8,56);
\draw[dashed,very thick,red] (0,0) parabola (8,64);
\draw[very thick,red]plot (\x,{(1/7)*\x^3});

\node at (1,7) {\tikzcircle[fill=white]{2.5pt}};
\foreach \x in {2,...,6}
	\node at (\x,\x*7) {\tikzcircle[fill=black]{2.5pt}};
\node at (8,56) {\tikzcircle[fill=black]{2.5pt}};

\node at (1,1) {\tikzcircle[fill=white]{2.5pt}};
\foreach \x in {2,...,6}
	\node at (\x,\x*\x) {\tikzcircle[fill=black]{2.5pt}};
\node at (8,64) {\tikzcircle[fill=black]{2.5pt}};

\node at (7,49) {\tikzcircle[fill=white,white]{2.75pt}}; 
\node at (7,49) {$\circlerighthalfblack$};

\node [above left] at (1,6.5) {\scriptsize(1,7)};
\node [above left] at (2,13.5) {\scriptsize(2,14)};
\node [above left] at (3,20.5) {\scriptsize(3,21)};
\node [above left] at (4,27.5) {\scriptsize(4,28)};
\node [above left] at (5,34.5) {\scriptsize(5,35)};
\node [above left] at (6,41.5) {\scriptsize(6,42)};

\node [above left] at (8,55.5) {\scriptsize(8,56)};

\node [above left] at (1,-1) {\scriptsize(1,1)};
\node [above left] at (2,2) {\scriptsize(2,4)};
\node [above left] at (3,7) {\scriptsize(3,9)};
\node [above left] at (4,14) {\scriptsize(4,16)};
\node [above left] at (5,23) {\scriptsize(5,25)};
\node [above left] at (6,34) {\scriptsize(6,36)};
\node [above left] at (7,48.5) {\scriptsize(7,49)};
\node [above left] at (8,62) {\scriptsize(8,64)};

\foreach \x in {1,...,8}
		\node [below] at (\x,-.75) {\x};
\foreach \x in {7,14,...,63,70}
		\node [left] at (-.075,\x) {\x};
\node at (-0.25,-0.25) {\includegraphics[width=5ex]{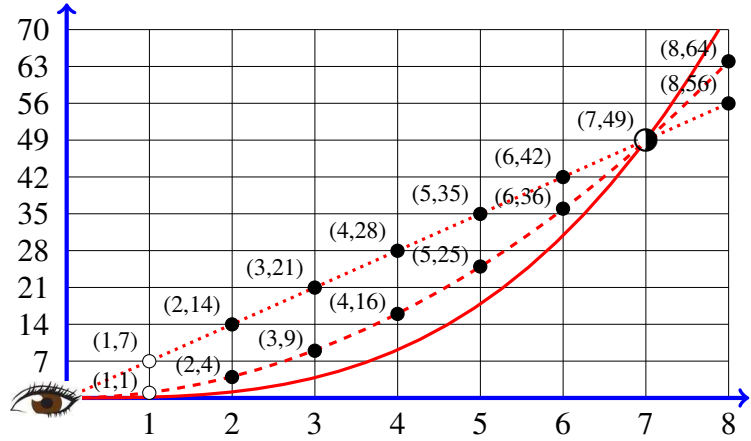}};
\end{tikzpicture}
\caption{Invisible and visible points under different lines of sights.}\label{fig:exampledifferentcurves}
\end{figure}
\end{example}

%%%%%%%%%%%%%%%%%%%%%%%%%%%%%%%%%%%%%%%%%%%%%%%%%%%%%%%%%%%%%%%%%%
%%%%%%%%%%%%%%%%%%%%%%%%%%%%%%%%%%%%%%%%%%%%%%%%%%%%%%%%%%%%%%%%%%
%%%%%%%%%%%%%%%%%%         SECTION 3       %%%%%%%%%%%%%%%%%%%%%%%
%%%%%%%%%%%%%%%%%%%%%%%%%%%%%%%%%%%%%%%%%%%%%%%%%%%%%%%%%%%%%%%%%%
%%%%%%%%%%%%%%%%%%%%%%%%%%%%%%%%%%%%%%%%%%%%%%%%%%%%%%%%%%%%%%%%%%

\section{Proportion of \texorpdfstring{$b$}{b}-visible lattice points}\label{Section3}

The literature on lattice point visibility presents rigorous proofs of the $b=1$ case of Theorem~\ref{thm:proportion_of_b_visible_points}, in particular in \textsc{Monthly} articles by Casey and Sadler~\cite[Theorem~1]{Casey} and Christopher~\cite[Theorem 1]{Christopher}. Other recent proofs (see~\cite{Abrams-Paris1992, Apostol2000}) give illuminating plausibility arguments but are merely heuristic sketches as there is no uniform probability distribution on the natural numbers and these arguments gloss over this important fact. However, these proofs can be made rigorous by the methods presented by Pinsky~\cite{Pinsky}. 
Following an analogous method, we now present a proof of our result regarding the proportion of $b$-visible points in the lattice, for $b\geq 1$.

\begin{proof}[Proof of Theorem~\ref{thm:proportion_of_b_visible_points}]
Fix $N, b \in \mathbb{N}$. Let $[N]:=\{1, 2, \ldots, N\}$. Let $r,s$ be two numbers picked independently with uniform probability in $[N]$ and fix a prime $p$ in  $[N]$. By Proposition~\ref{prop:visibility_criterion}, a point $(r,s)\in\mathbb{N}\times\mathbb{N}$ is $b$-visible if and only if $\ggcd_b(r,s)=1$. Let $P_N$ denote the probability that $p$ divides $r$ and $p^b$ divides $s$. There are $\left\lfloor{\frac{N}{p}}\right\rfloor$ integers in $[N]$ that are divisible by $p$; namely $p, 2p, \ldots, \left\lfloor{\frac{N}{p}}\right\rfloor p$. Thus the probability that $p$ divides $r$ is $\frac{1}{N}\left\lfloor{\frac{N}{p}}\right\rfloor$. Similarly, the probability that $p^b$ divides $s$ is $\frac{1}{N}\left\lfloor{\frac{N}{p^b}}\right\rfloor$. By mutual independence, the probability that $p$ divides $r$ and that $p^b$ divides $s$ is $P_N = \frac{1}{N^2} \left\lfloor{\frac{N}{p}}\right\rfloor \left\lfloor{\frac{N}{p^b}}\right\rfloor$. Therefore, the probability that $p$ does not divide $r$ or that $p^b$ does not divide $s$ is $1 - P_N$. Since $P_N \rightarrow \frac{1}{p^{b+1}}$ as $N \rightarrow \infty$, by multiplying over all of the primes we have that the probability that $p$ does not divide $r$ or that $p^b$ does not divide $s$ given that $p$ is prime is

\begin{equation*}\label{eqprime_new}
\displaystyle \lim_{N \rightarrow \infty} \prod_{\substack{p\;\text{prime}\\p\leq N}} \left( 1 - P_N \right) = 
\prod_{p\;\text{prime}}  \left(1-\frac{1}{p^{b+1}}\right)=\frac{1}{\zeta(b+1)},
\end{equation*}
where $\zeta(s)=\prod_{p\;\text{prime}}\left(1-1/p^s\right)^{-1}$.
\end{proof}

%%%%%%%%%%%%%%%%%%%%%%%%%%%%%%%%%%%%%%%%%%%%%%%%%%%%%%%%%%%%%%%%%%
%%%%%%%%%%%%%%%%%%%%%%%%%%%%%%%%%%%%%%%%%%%%%%%%%%%%%%%%%%%%%%%%%%
%%%%%%%%%%%%%%%%%%         SECTION 4       %%%%%%%%%%%%%%%%%%%%%%%
%%%%%%%%%%%%%%%%%%%%%%%%%%%%%%%%%%%%%%%%%%%%%%%%%%%%%%%%%%%%%%%%%%
%%%%%%%%%%%%%%%%%%%%%%%%%%%%%%%%%%%%%%%%%%%%%%%%%%%%%%%%%%%%%%%%%%

\section{Arbitrarily large \texorpdfstring{$b$}{b}-invisible forests}\label{Section4}

We exploit the Chinese remainder theorem to prove that arbitrarily large $m\times n$ arrays of adjacent  $b$-invisible integer lattice points in the plane exist for every $b \in \mathbb{N}$. We call such arrays of points $b$-invisible rectangular forests of size $m\times n$.

\begin{proof}[Proof of Theorem~\ref{thm:invisible_forests}]
It suffices to show that there exists a pair $(r,s)\in\mathbb{N}\times\mathbb{N}$ such that $\ggcd_b(r+i,s+j)\neq1$ for all $0\leq i<n$ and $0\leq j<m$.  To obtain a pair $(r,s)$, we first choose $mn$ distinct primes and label them $p_{i,j}$ where $0\leq i<n$ and $0\leq j<m$. Place the primes in a matrix as follows
$$
P_{m \times n} = 
 \begin{pmatrix}
  p_{0,m-1} & p_{1,m-1} & \cdots & p_{n-1,m-1} \\
  \vdots  & \vdots  & \iddots & \vdots  \\
  p_{0,1} & p_{1,1} & \cdots & p_{n-1,1} \\
  p_{0,0} & p_{1,0} & \cdots & p_{n-1,0} 
 \end{pmatrix}.
$$
The choice of the nonstandard indexing of the entries in the matrix $P_{m\times n}$ will become clear at the proof's conclusion. Set $C_i=\prod_{j=0}^{m-1}p_{i,j}$ and $R_j=\prod_{i=0}^{n-1}p_{i,j}$ and  
consider the following systems of linear congruences:

$$
\begin{cases}
\hspace{1cm}r+0 &\equiv 0 \pmod{C_0}\\
\hspace{1cm}r+1 &\equiv 0 \pmod{C_1}\\
&\hspace{1mm}\vdots  \\
r+(n-1) &\equiv 0 \pmod{C_{n-1}}
\end{cases}
\hspace{.3in}\mbox{and}\hspace{.3in}
\begin{cases}
\hspace{1cm}s+0 &\equiv 0 \pmod{R_0^b}\\
\hspace{1cm}s+1 &\equiv 0 \pmod{R_1^b}\\
&\hspace{1mm} \vdots  \\
s+(m-1) &\equiv 0 \pmod{R_{m-1}^b}.
\end{cases}
$$
The integers in the set $\{C_i\}_{i=0}^{n-1}$ are pairwise relatively prime.  Thus, by the Chinese remainder theorem, there exists a unique solution $r\pmod {\prod_{i=0}^{n-1} C_i$}.  Similarly the integers in the set $\{R_j\}_{j=0}^{m-1}$ are pairwise relatively prime and hence there is a unique solution $s\pmod {\prod_{j=0}^{m-1} R_j^b}$.

For each $0 \leq i < n$ and $0 \leq j < m$, we have by construction that $C_i$ divides $r+i$ and $R_j^b$ divides $s+j$, and thus $p_{i,j}$ divides $r+i$ and $p_{i,j}^b$ divides $s+j$. Hence $p_{i,j}$ divides $\ggcd_b(r+i, s+j)$ and so $\ggcd_b(r+i, s+j) \neq 1$. Hence every point $(r+i, \, s+j) \in \mathbb N \times \mathbb N$ with $0 \leq i < n$ and $0 \leq j < m$ is $b$-invisible, as desired.
\end{proof}

The proof of Theorem~\ref{thm:invisible_forests} constructs $b$-invisible forests of any dimension.  We illustrate this process below by constructing a $2$-invisible forest of size $2\times 3$.

\begin{example}\label{exam:2by3_forest} 
 Consider the prime matrix 
 $$P_{2 \times 3} = \left(
 \begin{array}{ccc}
 7 & 11 & 13\\
 2 & 3 & 5
 \end{array}
 \right).$$
Using the technique described in Theorem~\ref{thm:invisible_forests},
we compute the unique solution $r_0 \pmod{N}$ and $s_0 \pmod{N^2}$, where $N = 2 \cdot 3 \cdot 5 \cdot 7 \cdot 11 \cdot 13$, to the required system of linear congruences
\begin{align*}
r_0=r+0 &= 27 818 = 2 \cdot 7 \cdot 1987 \\
r_1=r+1 &= 27819 = 3^2 \cdot 11 \cdot 281 \\
r_2=r+2 &= 27820 = 2^2 \cdot 5 \cdot 13 \cdot 107\\
 s_0=s+0 &= 602202600 = 2^3 \cdot 3^5 \cdot 5^2 \cdot 12391\\
 s_1=s+1 &= 602202601 = 7^2 \cdot 11^2 \cdot 13^2 \cdot 601.
\end{align*}
The forest we have constructed is shown in Figure~\ref{tree23} with each corresponding value $\ggcd_2(r_i,s_j)$ noted in red.  One can easily verify that each of the six lattice points is $2$-invisible; indeed as the proof of Theorem~\ref{thm:invisible_forests} states, each prime $p_{i,j}$
in the prime matrix $P_{2 \times 3}$ divides the corresponding point $(r_i, s_j)$.

\begin{figure}[h]
\centering
\begin{tikzpicture}[scale=1.25]
\draw (0,0)--(0,1)--(2,1)--(2,0)--(0,0);
\draw (1,0)--(1,1);
\foreach \x in {0,1,2}
		\shade[ball color=blue] (\x,0) circle (.5ex);
\foreach \x in {0,1,2}
		\shade[ball color=blue] (\x,1) circle (.5ex);
\node [below] at (0,0) {\footnotesize $(r_0,s_0)$};
\node [above] at (0,1) {\footnotesize $(r_0,s_1)$};
\node [below] at (1,0) {\footnotesize $(r_1,s_0)$};
\node [above] at (1,1) {\footnotesize $(r_1,s_1)$};
\node [below] at (2,0) {\footnotesize $(r_2,s_0)$};
\node [above] at (2,1) {\footnotesize $(r_2,s_1)$};
\node [below] at (0,-.5) {\footnotesize \textcolor{red}{$2$}};
\node [above] at (0,1.5) {\footnotesize \textcolor{red}{$7$}};
\node [below] at (1,-.5) {\footnotesize \textcolor{red}{$3^2$}};
\node [above] at (1,1.5) {\footnotesize \textcolor{red}{$11$}};
\node [below] at (2,-.5) {\footnotesize \textcolor{red}{$2 \cdot 5$}};
\node [above] at (2,1.5) {\footnotesize \textcolor{red}{$13$}};
\end{tikzpicture}
\caption{A $2$-invisible forest of size $2\times3$.}\label{tree23}%$\mathcal{F}_{2,3}(27818, 602202600)$}
\end{figure}
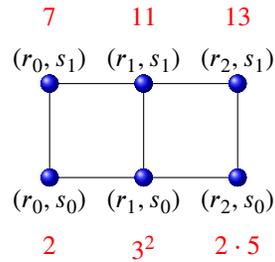
\end{example}

Although Theorem~\ref{thm:invisible_forests} provides a way to find $b$-invisible forests of an arbitrary size, it does not necessarily indicate which ones will be close to the origin. Finding the closest known invisible square forests (when $b=1$) was explored by Goodrich, Mbirika, and Nielsen~\cite{Mbirika2015}. In fact, using techniques from~\cite{Mbirika2015}, we find a closer hidden forest with $(r,s)=(440,38024)
$.  
An exhaustive computer implementation confirms that this is the closest $2$-invisible forest of size $2 \times 3$ in the first quadrant. We end by posing the following $b$-visibility problem:
For fixed values $b,n,m\in\mathbb{N}$, find the nearest $b$-invisible forest of dimension $n\times m$.

%--------------------------------------

\begin{footnotesize}

\bigskip
\bigskip

\noindent
\textbf{\uppercase{Acknowledgments.}} We thank Stephan Garcia for helpful references and for a conversation at REUF4 at ICERM in Summer 2012 which motivated the fourth author's interest in lattice point visibility. We also thank the undergraduate UW-Eau Claire research students Austin Goodrich, Jasmine Nielsen, Michele Gebert, and Sara DeBrabander who studied lattice point visibility both in the classic and generalized cases with us.

\bigskip
\bigskip

\noindent
\textbf{\uppercase{Edray Herber Goins}} grew up in South Los Angeles, California.  He works in the field of number theory, as it pertains to the intersection of representation theory and algebraic geometry.   He is currently the president of the National Association of Mathematics.\\
\textit{Department of Mathematics, Purdue University, West Lafayette IN 47906}\\
{\tt egoins@purdue.edu}

\bigskip

\noindent
\textbf{\uppercase{Pamela E. Harris}} received her PhD degree in mathematics from the University of Wisconsin Milwaukee, held a postdoctoral position at the United Stated Military Academy, and is now an Assistant Professor of Mathematics at Williams College.  Her research interests are in algebra and combinatorics, particularly as these subjects relate to the representation theory of Lie algebras. Her service commitments aim to increase the visibility of Latinx and Hispanic mathematicians via platforms such as \url{http://www.lathisms.org}.\\
\textit{Department of Mathematics and Statistics, Williams College, Williamstown MA 01267}\\
{\tt peh2@williams.edu}

\bigskip

\noindent
\textbf{\uppercase{Bethany Kubik}} received her PhD from North Dakota State University, held a postdoctoral position at the United States Military Academy, and is now an Assistant Professor at University of Minnesota Duluth. Her research interests include homological algebra, factorization, and graph theory.\\
\textit{Department of Mathematics and Statistics, University of Minnesota Duluth, Duluth MN 55812}\\
{\tt bakubik@d.umn.edu}

\bigskip

\noindent
\textbf{\uppercase{Aba Mbirika}} received his PhD from University of Iowa, held a postdoctoral position at Bowdoin College, and very recently in Fall 2017 became an Associate Professor of Mathematics at University of Wisconsin-Eau Claire. His mathematical interests include combinatorial representation theory, complex reflection groups, lattice point visibility, and graph labelings. aBa (his preferred spelling) does not know how to drive, but he happily bikes even in the Wisconsin winters. His answer to the question ``Where are you from?'' is always Iowa because in his graduate school years there, Iowa City became his favorite place on Earth.\\
\textit{Department of Mathematics, University of Wisconsin-Eau Claire, Eau Claire WI 54701}\\
{\tt mbirika@uwec.edu}

\end{footnotesize}

\end{document}